\documentclass[12pt,a4paper]{tac}

\usepackage{amsmath, amssymb}

\usepackage{mathrsfs}
\usepackage{hyperref}

\usepackage{quiver}
\usepackage{tikz-cd}
\usepackage[all,2cell]{xy}

\usepackage{todonotes}

\DeclareMathAlphabet{\mathpzc}{OT1}{pzc}{m}{it}

\def\dotminus{\mathbin{\ooalign{\hss\raise1ex\hbox{.}\hss\cr
  \mathsurround=0pt$-$}}}

\def\sR{\mathsf{R}_{\mS}}

\def\ED{\mathsf{ED}}
\def\PD{\mathsf{PD}}

\def\theory{\mathbb{T}}

\newcommand{\angbr}[2]{\langle #1,#2 \rangle} 

\newcommand{\freccia}[3]{#2\colon#1 \to #3}

\newcommand{\frecciainj}[3]{\xymatrix{#2 \colon #1  \ar@{^{(}->}[r] &  #3}}

\newcommand{\duefreccia}[3]{\xymatrix@C=0.5cm{#2 \colon #1  \ar@{=>}[r] &  #3}}

\newcommand{\duemorfismo}[6]{\xymatrix@+1pc{
#1^{\op} \ar[rrd]^#2_{}="a" \ar[dd]_{#3^{\op}}\\
&& \infsl\\
#5^{\op}  \ar[rru]_#6^{}="b"
\ar_{#4}  "a";"b"}}
\newcommand{\comsquare}[8]{ \xymatrix@+1pc{ 
#1 \ar[r]^{#5} \ar[d]_{#6} & #2 \ar[d]^{#7} \\
#3 \ar[r]_{#8} & #4 
}}
\newcommand{\pullback}[8]{ \xymatrix@+1pc{ 
#1 \pullbackcorner \ar[r]^{#5} \ar[d]_{#6} & #2 \ar[d]^{#7} \\
#3 \ar[r]_{#8} & #4 
}}
\newcommand{\quadratocomm}[8]{ \xymatrix@+1pc{ 
#1 \ar[r]^{#5} \ar[d]_{#6} & #2 \ar[d]^{#7} \\
#3 \ar[r]_{#8} & #4 
}}
\newcommand{\comsquarelargo}[8]{ \xymatrix@+1pc{ 
#1 \ar[rr]^{#5} \ar[d]_{#6} && #2 \ar[d]^{#7} \\
#3 \ar[rr]_{#8} && #4 
}}
\newcommand{\parallelmorphisms}[4]{\xymatrix@+1pc{
#1 \ar @<+4pt>[r]^{#2} \ar @<-4pt>[r]_{#3} & #4
}}
\newcommand{\relation}[4]{\xymatrix@+1pc{
\angbr{#2}{#3}\colon #1 \ar @<+4pt>[r] \ar @<-4pt>[r] & #4
}}
\newcommand{\frecceparalleleopposte}[4]{\xymatrix@+1pc{
#1 \ar@<+4pt>[r]^{#2} \ar@<-4pt>@{<-}[r]_{#3} & #4
}}
\newcommand{\equalizer}[6]{\xymatrix@+1pc{
#1 \ar[r]^{#2} & #3 \ar @<+4pt>[r]^{#4} \ar @<-4pt>[r]_{#5} & #6
}}
\newcommand{\coequalizer}[6]{\xymatrix@+1pc{
 #1 \ar @<+4pt>[r]^{#2} \ar @<-4pt>[r]_{#3} & #4 \ar[r]^{#5} & #6
}}

\newcommand{\sottoggetto}[2]{\xymatrix{
#1 \ar@{>->}[r] & #2
}}

\newcommand{\pullbackcorner}[1][ul]{\save*!/#1+1.2pc/#1:(1,-1)@^{|-}\restore}

\def\pr{\pi}
\def\id{\operatorname{ id}}
\def\op{\operatorname{ op}}

\def\mA{\mathcal{A}}

\def\mC{\mathcal{C}}

\def\mD{\mathcal{D}}

\def\mG{\mathcal{G}}

\def\mN{\mathsf{Nat}}

\def\mS{\mathcal{S}}

\def\Sub{\mathsf{Sub}}

\def\infsl{\mathsf{InfSl}}

\def\set{\mathsf{Set}}
\newcommand{\cont}{\mathcal{V}}

\newcommand{\EED}{\mathsf{EED}}

\newcommand{\Excat}{\mathsf{ExCat}}
\newcommand{\regdoctrine}[1]{\mathsf{Reg}(#1)}
\newcommand{\exdoctrine}[1]{\mathsf{Ex}(#1)}
\newcommand{\Reg}{\mathsf{RegCat}}

\def\Pred#1{{\ensuremath{{\mathpzc{Prd}\kern-.4ex_{{#1}}}}}}

\newcommand{\doctrine}[2]{#2\colon #1^{\op}\longrightarrow\infsl}

\def\reg{\mathsf{reg}}
\def\ex{\mathsf{ex}}
\def\lex{\mathsf{lex}}
\newcommand{\exlex}[1]{(#1)_{\mathsf{ex}/\mathsf{lex}}}
\newcommand{\exreg}[1]{(#1)_{\mathsf{ex}/\mathsf{reg}}}
\newcommand{\reglex}[1]{(#1)_{\mathsf{reg}/\mathsf{lex}}}

\newcommand{\purecompex}[1]{{#1}^{\exists}}

\def\lang{\mathcal{L}}
\def\syntdoc{\mathsf{LT}}
\title{Quotients, pure existential completions and Arithmetic Universes}
\author{Maria Emilia Maietti and Davide Trotta}

\begin{document}

\maketitle

\begin{abstract}
We provide a new description of Joyal's arithmetic universes through a  characterization of the exact and regular completions of pure existential completions.

We show that the regular and  exact completions of  the pure existential completion  of an  elementary doctrine $P$ are equivalent to the  $\reg/\lex$ and $\ex/\lex$-completions, respectively, of the category of predicates of $P$.

This result generalizes a previous one  by the first author  with F.~Pasquali and G.~Rosolini  about doctrines equipped with Hilbert's $\epsilon$-operators.

Thanks to  this characterization,  each arithmetic universe in the sense of Joyal  can be seen  as the  exact completion of the pure existential completion of the doctrine of predicates of its Skolem theory.

  In particular,  the  initial arithmetic universe in the standard category of ZFC-sets turns out to be the completion with exact quotients of the doctrine of recursively enumerable predicates.
\end{abstract}

\begin{center}
\begin{tabular}{r}
{\it Dedicated  to Pieter Hofstra for  his  inspiring creative work.}
\end{tabular}
\end{center}

\section{Introduction}
This paper provides a new contribution to the description of Joyal's arithmetic universes  thought the application of a new characterization  of free completions of elementary Lawvere doctrines.

Free completions  of categories with quotients are ubiquitous in category theory. In particular those leading to exact and regular categories  in~\cite{SFEC,REC} have been widely studied in the literature of  category theory, with applications both
to mathematics and computer science, see~\cite{Lack1999,phdMenni,MENNI2002187,QCFF,VANDENBERG20183137,VANDENBERG2008123,maietti_maschio_2021}.

Such free completions are also involved in the construction of arithmetic universes introduced by A.~Joyal  to prove G{\"o}del incompleteness theorems in some lectures (still unpublished) in the seventies and recalled in~\cite{Joyal} (see~\cite{JAULAP,vandijk2020godel} for more information).  A more general abstract definition of arithmetic universe as list-arithmetic pretopos has been proposed  in~\cite{JAULAP}.

In more detail,  an arithmetic universe in the sense of  Joyal  can be described as the exact completion -- in the sense of~\cite{SFEC} -- of the  lex category of predicates of a given Skolem theory. In particular, the initial arithmetic universe within the standard category of ZFC-sets turns out to be the exact completion of the  lex category of primitive recursive predicates. 
In recent years, the regular and exact completions of a lex category have been proved to be instances of more general completions of certain Lawvere doctrines in \cite{QCFF,EQC,UEC}.

More precisely, the exact completion of a lex category~\cite{SFEC,REC}, also known as $\ex/\lex$-completion, and the exact completion of a regular category, referred to as $\ex/\reg$-completion,  have been proven to be an instance of a more general  exact completion  $\exdoctrine{P}$  relative to an  elementary, pure existential doctrine $P$.

The construction of the exact category $\exdoctrine{P}$  is, essentially, achieved through the tripos-to-topos construction developed by J.M.E.~Hyland, P.T.~Johnstone and A.M.~Pitts in~\cite{TT,TTT}, applied to an elementary existential doctrine.
Similarly, the regular completion of a lex category, also called $\reg/\lex$-completion,  introduced in~\cite{FECLEO,REC} has been proved to be a specific instance  of a more general construction for doctrines in~\cite{TECH}, which is the regular completion $\regdoctrine{P}$ of an elementary, pure  existential doctrine $P$.  
The regular and exact completions of doctrines are related via the $\ex/\reg$-completion,  namely we have that  $\exdoctrine{P}\equiv \exreg{\regdoctrine{P}}$~\cite{UEC}.

In this paper we provide a characterization which further relates the exact  and regular completions  of a lex category 
with the notion of  exact completion  of an  elementary existential doctrine  as presented in~\cite{QCFF,EQC,UEC} by involving
a third kind of free completion, namely the \emph{pure existential completion} of an elementary doctrine.

The notion of \emph{existential completion}  was introduced by the second author in his PhD thesis and later published in~\cite{ECRT}, and it is a construction that freely adds existential quantifiers to a given primary doctrine, along with a class $\Lambda$ of base morphisms that is closed under compositions, pullbacks and isomorphisms. To distinguish some particularly interesting instances, this free construction has been renamed  as {\it generalized existential completion} in~\cite{maietti_trotta}.
Following the terminology in~\cite{maietti_trotta}, the \emph{pure existential completion} is
the instance of the generalized existential completion  where the class $\Lambda$ consists of product projections,
   while  we call \emph{full existential completion}  the instance  where $\Lambda$ is the class of all the base morphisms.
Furthermore, the pure existential completion  also coincides with the restriction to faithful fibrations of the \emph{simple coproduct (or sum) completion} of a fibration employed by P.~Hofstra in~\cite{hofstra2011}, while
 the full existential completion coincides with the $\exists$-completion introduced  by  J.~Frey in~\cite{Frey2020,Frey2014AFS} and a particular case of it has  also  played a significant role in the works by P.~Hofstra~\cite{Hofstra06,PHDHofstra}.

In this work, we show that the regular and  exact completions of  the pure existential completion  of an  elementary doctrine $P$ are equivalent to the $\reg/\lex$ and $\ex/\lex$-completions, respectively, of the category of predicates of $P$.

In detail, we show that  for an elementary, pure existential doctrine $P$ and an elementary subdoctrine $P'$ of $P$ on the same base category,
the regular completion $\regdoctrine{P}$ of $P$ corresponds to the $\reg/\lex$-completion  $\reglex{\Pred{P'}}$  of the category of predicates $\Pred{P'}$ of $P'$
(via an equivalence  induced by the canonical embedding of   $\Pred{P'}$ into $\regdoctrine{P}$) 
 if and only if $P$ is the pure existential completion of $P'$ (Theorem~\ref{thm_main_1}).  
 
 Then,  by combining this result with the aforementioned decomposition of exact completions, we immediately deduce 
 that the exact completion $\exdoctrine{P}$ of $P$ corresponds to the  $\ex/\lex$-completion  $ \exlex{\Pred{P'}}$  of the category of predicates $\Pred{P'}$ of $P'$ 
(again via an equivalence induced by the canonical embedding of   $\Pred{P'}$ into $\exdoctrine{P}$) 
  if and only if $P$ is the pure existential completion of $P'$ (Corollary~\ref{cor_ex_comp_main_2}).

 The crucial intuition is that there is a tight connection between \emph{regular projective objects} of $\regdoctrine{P}$ and \emph{pure existential free elements} of  $P$. In particular, we show that whenever $P$ is the pure existential completion of $P'$ (and hence the elements of $P'$ are pure existential free elements according to~\cite{maietti_trotta}), we can use $P'$ to define a projective cover for $\regdoctrine{P}$  which in addition satisfies the property that  every object of $\regdoctrine{P}$ can be embedded into a projective of this  cover.

Our  characterization generalizes a previous one in~\cite{TECH}  by the first author  with F.~Pasquali and G.~Rosolini  about doctrines equipped with Hilbert's $\epsilon$-operators 
for the fact that  in~\cite{maietti_trotta} we showed that  a doctrine is equipped with Hilbert's $\epsilon$-operators if and only if it is equivalent to the pure existential completion of itself.

Then, we apply our characterization to deduce that each arithmetic universe turns out to be the  exact completion  -- in the sense  of~\cite{UEC} --  of the pure existential completion of the doctrine of predicates of a given Skolem theory.
As a consequence we deduce that the  initial arithmetic universe in the standard category of ZFC-sets is the completion with exact quotients of the doctrine of recursively enumerable
predicates.

Another notable application of our characterization,  already published in~\cite{trotta23TCS} by  employing our main theorem first presented in~\cite{maietti_trotta_arxiv} (with a different proof), regards the exact completion of G\"odel hyperdoctrines introduced in \cite{trotta23TCS,trotta2022}  as an equivalent presentation of the restriction to faithful fibrations of Hofstra's Dialectica fibrations~\cite{hofstra2011}.

\section{Elementary and existential doctrines}
The term {\it doctrine}, when accompanied by certain adjectives, is often associated with a generalization of the concept of {\it hyperdoctrine} introduced by F.W.~Lawvere in a series
of seminal papers \cite{AF,EHCSAF,DACCC}. We recall from \textit{loc. cit.}
some definitions which will be useful in the following. The reader can
find all the details about the theory of elementary and pure existential
doctrines also in \cite{EQC,QCFF,UEC,TECH}, and an algebraic analysis of the elementary structure of a doctrine in \cite{EMMENEGGER2020}.

\begin{definition}\label{def primary doctrine}
A \textbf{primary doctrine}~is a functor
$\doctrine{\mC}{P}$ from the opposite of a category $\mC$ with finite products to the category $\infsl$ of inf-semilattices.
\end{definition}
We will use the notation $\alpha\wedge \beta$ to denote the binary inf of $\alpha$ and $\beta$ in $P(A)$ and $\top_A$ to denote the top element of $P(A)$.

\begin{definition}\label{def elementary doctrine}
A primary doctrine $\doctrine{\mC}{P}$   is \textbf{elementary} if for every $A$ in $\mC$ there exists an object $\delta_A$ in $P(A\times A)$, called \textbf{fibered equality}, such that
\begin{enumerate}
\item the assignment
\[\exists_{\angbr{\id_A}{\id_A}}(\alpha):=P_{\pr_1}(\alpha)\wedge \delta_A\]
for an element $\alpha$ of $P(A)$ determines a left adjoint to $\freccia{P(A\times A)}{P_{\angbr{\id_A}{\id_A}}}{P(A)}$;
\item for every morphism $e$ of the form $\freccia{X\times A}{\angbr{\pr_1,\pr_2}{\pr_2}}{X\times A\times A}$ in $\mC$, the assignment
\[ \exists_{e}(\alpha):= P_{\angbr{\pr_1}{\pr_2}}(\alpha)\wedge P_{\angbr{\pr_2}{\pr_3}}(\delta_A)\]
for $\alpha$ in $P(X\times A)$ determines a left adjoint to $\freccia{P(X\times A \times A)}{P_e}{P(X\times A)}$.
\end{enumerate}
\end{definition}
\begin{example}\label{example horn fragment}
    Let $\lang_{=}$ be the $(\top,\wedge,=)$-fragment of  Intuitionistic Logic, i.e. the fragment with true constant, conjunctions and equality, called Horn-fragment in \cite[Sec. D.1.1, p. 810]{SAE}. Let  $\theory$ be a theory in such a fragment.  Let us denote by $\cont$ the syntactic category whose objects are contexts (up to $\alpha$-equivalence), and arrows are term substitutions. Consider the functor
    \[\doctrine{\cont}{\syntdoc_{=}^{\theory}}\]
     defined on a given context $\Gamma$  of  $\cont$ by taking $\syntdoc_{=}^{\theory}(\Gamma)$ as the Lindenbaum-Tarski algebra of well-formed formulas
    of $\lang_{=}$ with free variables in $\Gamma$  and on a substitution morphism between contexts by taking  the substitution homomorphism between  formulas of
    the Lindenbaum-Tarski algebras. The functor $\syntdoc_{=}^{\theory}$ is an elementary doctrine.

\end{example}

\begin{definition}\label{def pure existential doctrine}
  A primary doctrine $\doctrine{\mC}{P}$ is \textbf{pure existential} if, for every object $A$ and $B$ in $\mC$ for any product projection $\freccia{A\times B}{\pr_A}{A}$, the functor
  \[ \freccia{P(A)}{{P_{\pr_A}}}{P(A\times B)}\]
  has a left adjoint $\exists_{\pr_A}$, and these satisfy:
    \begin{enumerate}
    \item[(BCC)] \textbf{Beck-Chevalley condition:} for any pullback diagram
    \[\begin{tikzcd}[column sep=large, row sep=large]
        {C\times B} & {C} \\
        A\times B & A
        \arrow["{f\times \id_B}"', from=1-1, to=2-1]
        \arrow["\pr_A"', from=2-1, to=2-2]
        \arrow["{\pr_C}", from=1-1, to=1-2]
        \arrow["f", from=1-2, to=2-2]
        \arrow["\scalebox{1.6}{$\lrcorner$}"{anchor=center, pos=0.1}, shift left=3, draw=none, from=1-1, to=2-1]\end{tikzcd}\]
    the canonical arrow
    \[ \exists_{\pr_C}P_{f\times_{\id_B}}(\alpha)\leq P_f \exists_{\pr_A}(\alpha)\]
    
    is an isomorphism for every element $\alpha$ of the fibre $P(A\times B )$;
    \item[(FR)] \textbf{Frobenius reciprocity:} for any projection $\freccia{A\times B}{{\pr_A}}{A}$, for any object $\alpha$ in $P(A)$ and $\beta$ in $P(X\times A)$, the canonical arrow
    \[ \exists_{\pr_A}(P_{\pr_A}(\alpha)\wedge \beta)\leq \alpha \wedge \exists_{\pr_A}(\beta)\]
    in $P(A)$ is an isomorphism.
  \end{enumerate}
\end{definition}

\textbf{Notation:} in this work, given two primary doctrines $P$ and $P'$ on the same base category $\mC$, we will say that $P'$ is a \textbf{subdoctrine} of $P$ if $P'(X)$ is  a sub-inf-semilattice of $P(X)$ for every object $X$ of $\mC$, and if, for every arrow $\freccia{Y}{f}{X}$, the action of $P_f$ and $P_f'$ is the same (on the object of $P'(X)$). Moreover, we say that $P'$ is an \textbf{elementary subdoctrine} of $P$ whenever both  doctrines are elementary and also the fibred equality of $P'$ coincides with that of $P$.

\begin{example}\label{example regular fragment}
    Let $\lang_{=,\exists}$ be the $(\top,\wedge,=,\exists)$-fragment of first-order Intuitionistic Logic (also called \emph{regular} in~\cite[Sec. D1.3]{SAE}, see remark~\ref{regularlogicwrongname}), i.e. the fragment with the true constant, conjunction, equality and existential quantifiers. Let $\theory$ be a theory of such a fragment. 
    Consider the syntactic doctrine 
    $$\doctrine{\cont}{\syntdoc_{=,\exists}^{\theory}}$$
    where $\cont$ is the category of contexts and substitutions and $\syntdoc_{=,\exists}^{\theory}(\Gamma)$ is  given by the Lindenbaum-Tarski algebra of well-formed formulas
    of $\lang_{=,\exists }$ with free variables in $\Gamma$  as in Example~\ref{example horn fragment}.  The  doctrine $\syntdoc_{=,\exists}^{\theory}$  is elementary and pure existential.
\end{example}

\begin{remark}\label{rem left adjoint elementary pure existential doc}
    In a pure existential elementary doctrine, for every arrow $\freccia{A}{f}{B}$ of $\mC$ the functor $P_f$ has a left adjoint $\exists_f$ that can be computed as
    \[ \exists_{\pr_2}(P_{f\times {\id}_B}(\delta_B)\wedge P_{\pr_1}(\alpha))\]
    for $\alpha$ in $P(A)$, where $\pr_1$ and $\pr_2$ are the projections from $A\times B$.
    However, observe that such a  definition  guarantees only the validity of the corresponding Frobenius reciprocity condition for $\exists_f$, but it does not guarantee the validity of the  Beck-Chevalley condition with respect to pullbacks along $f$ (see the counterexample in \cite[Rem. 6.4]{maietti_trotta}). In particular, primary doctrines, whose base category has finite limits, having left adjoints along every morphisms satisfying BCC and FR are called \emph{full existential} in \cite{maietti_trotta}.
\end{remark}

The following examples are discussed in \cite{AF,FSF}.

\begin{example}\label{example subobjects doctrine}
 Let $\mC$ be a category with finite limits. The subobject functor $\doctrine{\mC}{{\Sub_{\mC}}}$ is an elementary doctrine. Moreover,  it is pure existential if and only if the category $\mC$ is regular.
\end{example}
\begin{example}\label{example weak subobjects doctrine}
Let $\mD$ be a category with finite products and weak pullbacks. The weak subobjects (or variations) functor $\doctrine{\mD}{{\Psi_{\mD}}}$, assigning to an object $A$ the poset reflection of the slice category $\mD/A$, is an elementary and pure existential doctrine (left adjoints are given by the post-composition). Moreover, we know from \cite{maietti_trotta} that  every weak subobject doctrines is a full existential completion (and that every element of the fibre can be written as an existential quantifier of a top element).
\end{example}

The category of  primary doctrines $\PD$ is a 2-category, and we refer to \cite{QCFF,EQC} for a complete description of the 1-cells and 2-cells of this 2-category.
We denote by $\ED$ the 2-full subcategory of $\PD$ whose objects are pure existential doctrines, and whose 1-cells are those 1-cells of $\PD$ which preserve the pure existential structure. Similarly, we denote by $\EED$ the 2-full subcategory of $\PD$ whose objects are elementary and pure existential doctrines, and whose 1-cells preserve both the pure existential and the elementary structure.

We conclude this section recalling from \cite{EHCSAF,EQC,QCFF} the   \textbf{Grothendieck category} and the \textbf{category of predicates}  of an elementary doctrine.

\begin{definition}\label{def comprehension comp}
  Given a primary doctrine $\doctrine{\mC}{P}$,  its \textbf{Grothendieck category} $\mG_P$ is defined as follows:
  \begin{itemize}
  \item an object of $\mG_P$ is a pair $(A,\alpha)$ where $A$ is a set and $\alpha\in P(A)$;
  \item an arrow $\freccia{(A,\alpha)}{f}{(B,\beta)}$ is an arrow $\freccia{A}{f}{B}$ of $\mC$ such that $\alpha\leq P_f(\beta)$.
  \end{itemize}
  \end{definition}
 We just remind that the Grothendieck category $\mG_P$  of $P$ is the base of the free completion adding \emph{comprehensions} to $P$ as shown in \cite{EQC,UEC}.

  \begin{definition}\label{def pred cat}
    Given an elementary doctrine $\doctrine{\mC}{P}$,  its \textbf{category of predicates} $\Pred{P}$ is defined as follows:
    \begin{itemize}
    \item an object of $\Pred{P}$  is a pair $(A,\alpha)$ where $A$ is a set and $\alpha\in P(A)$;
    \item an arrow $\freccia{(A,\alpha)}{[f]}{(B,\beta)}$ of $\Pred{P}$ is an equivalence class of arrows $(A,\alpha)\to (B,\beta)$  of $\mG_P$ with respect to the following equivalence relation: $f\sim g$ when $\alpha\leq P_{\angbr{f}{g}}(\delta_B)$.
    \end{itemize}
    \end{definition}
 We just remind  that  the category of predicates  $\Pred{P}$  of $P$ is the base of the free completion adding  an extensional equality, formally named  {\it comprehensive diagonals} (see \cite[Def. 5.2]{EQC}),  to the free completion of  $P$  with comprehensions, as shown in \cite{EQC}.

    \begin{remark}\label{rem_finite_limit_pred}
      Notice that the category of predicates $\Pred{P}$ of an elementary doctrine has always finite limits. We refer to \cite[Prop.4.15]{QCFF} or \cite[Rem. 2.14]{TECH} for the explicit description of the pullbacks in $\Pred{P}$, being the base of an elementary doctrine with full comprehensions and comprehensive diagonals. The name {\it category of predicates} in \cite{TECH}  was inspired by Joyal's category of predicates in \cite{JAULAP} (see \cite[Ex. 2.18]{TECH}) which we will recall in the last section.
      \end{remark}

\section{The pure existential completion}
In \cite{ECRT} the second author introduced a free construction, called \emph{existential completion}, that freely adds left adjoints along a given class of morphisms $\Lambda$ (closed under pullbacks, compositions and isomorphisms) to a given primary doctrine.  Such a notion has been renamed {\it generalized existential completion} in \cite{maietti_trotta} to distinguish some of its relevant instances. Following the terminology in \cite{maietti_trotta}, the \emph{pure existential completion} is
the instance of the generalized existential completion  where the class $\Lambda$ consists of product projections.

In \cite{maietti_trotta} we provided an intrinsic characterization of generalized existential completions through the notion of {\it existential free elements} with respect to $\Lambda$.  Here, we present a further version of this characterization (Theorem~\ref{theorem caract. gen. pure existential completion}) only for the pure existential completion (but it can be smoothly extended to all generalized existential completions) by introducing  the notion of \emph{pure existential free  elements  of an existential doctrine $P$ relative to a  subdoctrine $P'$} which slightly generalizes  the notion of \emph{pure existential free  objects  of a doctrine $P$}. This result will be relevant in the proof of Theorem~\ref{thm_main_1}.

Since we will mainly perform our calculations on pure existential completions by just referring to their intrinsic characterizations, here we do not recall the original construction from \cite{ECRT}. We just remind  from \cite{ECRT} that the pure existential completion  $\doctrine{\mC}{\purecompex{P }}$ of a given primary doctrine $\doctrine{\mC}{P}$  provides  a 2-adjunction
\[\begin{tikzcd}
	\PD && \ED
	\arrow[""{name=0, anchor=center, inner sep=0},"(-)^{\exists}", curve={height=-18pt}, from=1-1, to=1-3]
	\arrow[""{name=1, anchor=center, inner sep=0},hook', curve={height=-18pt}, from=1-3, to=1-1]
	\arrow["\dashv"{anchor=center, rotate=-90}, draw=none, from=0, to=1]
\end{tikzcd}\]
from the 2-category $\PD$ of primary doctrines into the 2-category $ \ED$ of pure existential doctrines \cite{ECRT} and this justifies its name. Moreover, the pure existential completion preserves the elementary structure as shown in \cite{ECRT} and this is necessarily so as shown in \cite[Thm. 6.1]{maietti_trotta}.

Now, we start by giving the main definitions necessary for  the intrinsic characterization of  the pure existential completion.

For the rest of this section, let $\doctrine{\mC}{P}$ be a fixed pure existential doctrine, and let $\doctrine{\mC}{P'}$ be a fixed subdoctrine of $P$. 
    
    \begin{definition}\label{definition pure ex splitting}\label{definition pure ex free}
    An element $\alpha$ of the fibre $P(A)$ is said to be a \textbf{pure existential splitting} if 
    for every projection $\freccia{A\times B}{\pr_A}{A}$  and for every element $\beta$ of the fibre $P(A\times B)$, whenever $ \alpha = \exists_{\pr_A}(\beta)$ holds
    then there exists an arrow $\freccia{A}{h}{B}$ such that
    $   \alpha = P_{\angbr{\id_A}{h}}(\beta) $. Moreover, $\alpha$ is said to be \textbf{pure existential free} if for every morphism $\freccia{B}{f}{A}$, $P_f(\alpha)$ is  a pure existential splitting.
    \end{definition}
In the following proposition we recall from \cite[Prop. 4.4]{maietti_trotta} a useful equivalent characterization of pure existential splitting elements that is a strengthening of the usual Existence Property:
    \begin{proposition} \label{proposition:existential free elements}
Let $\alpha$ be an element of the fibre $P(A)$. Then  $\alpha$ is pure existential splitting if and only if for every projection $\freccia{A\times B}{\pr_A}{A}$ 
  and for every element $\beta$ of the fibre $P(A\times B)$, whenever $ \alpha \leq \exists_{\pr_A}(\beta)$ holds   then there exists an arrow $\freccia{A}{h}{B}$ such that   $   \alpha \leq P_{\angbr{\id_A}{h}}(\beta) $.

\end{proposition}

    \begin{definition}\label{definition RC}
        $P$ satisfies the \textbf{Rule of Choice}, for short {\bf  (RC)},  if whenever
        $\top_A\leq \exists_{\pr_A} (\beta) $
        there exists an arrow $\freccia{A}{h}{B}$ such that
        $\top_A\leq P_{\angbr{\id_A}{h}}(\beta)$.
        
        \end{definition}
        
        \begin{remark}\label{rem_RC_and top_ex_splitting}
        Observe that $P$ satisfies (RC) if and only if for every object $A$ of $\mC$, the element $\top_A\in P(A)$ is a pure existential splitting. 
        \end{remark}
           \begin{definition}
We say that an element $\alpha $ of the fibre $ P(A)$ is \textbf{covered} by an element $\beta\in P(A\times B)$ if  $\alpha=\exists_{\pr_A}(\beta)$.
\end{definition}

\begin{definition}\label{def enough pure existential free}
  We say that $P$ has \textbf{enough pure existential free elements} if for every object $A$ of $\mC$, any element $\alpha\in P(A)$ is covered by some pure existential free element $\beta\in P(A\times B)$ for some object $B$ of $\mC$.   
\end{definition}

\begin{remark}\label{enough_PEF_implies_every_PES_is_PEF}
  Observe that  if all elements of  a subdoctrine $P'$ 
  of a pure existential doctrine $P$
  are pure existential splitting for $P$  then they are also pure existential free for $P$, being the doctrine $P'$ closed under re-indexing.
  It also holds that,   if a doctrine    has enough pure existential free elements then every pure existential splitting is pure existential free. We refer to \cite[Lem. 4.11]{maietti_trotta} for a proof of this fact.
\end{remark}
       
\begin{definition}\label{def:lambda existential cover}
  We say that $P'$ is  a \textbf{pure existential cover} of $P$ if  for any object $A$ of $\mC$, every element $\alpha'$ of $P'(A)$ is a pure existential splitting for $P$ (and hence pure existential free) and  every element  $\alpha$ of $P(A)$ is  covered by an element of $P'$.
\end{definition}

We summarize in the following proposition some useful properties, and refer to \cite{maietti_trotta} for all the details:

\begin{lemma}\label{lemma_cover_P_and_ex_excomp}
The following two results hold:
  \begin{itemize}
    \item if $P$ is the pure existential completion of a primary doctrine $P'$ then $P'$ is a pure existential cover of $P$;
    \item if $P'$ is a pure existential cover of $P$, then the existential free elements of $P$ coincides exactly with the elements of $P'$. Hence, if a pure existential cover exists, it is unique.
  \end{itemize}
\end{lemma}

The previous notions can be generalized by relativizing each concept to a given subdoctrine:
 \begin{definition}\label{definition pure ex splitting rel to P}\label{definition pure ex free rel to P}
  An object $\alpha$ of the fibre $P'(A)$ is said to be a \textbf{pure existential splitting of $P$ relative to} $P'$ if 
  for every projection $\freccia{A\times B}{\pr_A}{A}$  and for every element $\beta$ of the fibre $P'(A\times B)$, whenever $ \alpha = \exists_{\pr_A}(\beta)$ holds in $P(A)$
  then there exists an arrow $\freccia{A}{h}{B}$ such that
  $   \alpha = P_{\angbr{\id_A}{h}}(\beta) $. Moreover, $\alpha$ is said to be \textbf{pure existential free of $P$ relative to} $P'$ if $P_f(\alpha)$ is  a pure existential splitting of $P$ relative to $P'$ for every morphism $\freccia{B}{f}{A}$.
\end{definition}

\begin{definition}\label{def:relative existential cover}
   We say that $P'$ is  a \textbf{pure existential relative cover of $P$}  if  for any object $A$, every element $\alpha'$ of $P'(A)$ is a pure existential splitting element of $P$ relative to $P'$ and  every element  $\alpha$ of $P(A)$ is  covered by an element of $P'$.
\end{definition}

\begin{lemma}\label{lemma_relative_splitting_and_cover_implies_splitting}
  If every element of $P$ is covered by an element of $P'$,  then every pure existential splitting of $P$ relative to $P'$ is a pure existential splitting. 
\end{lemma}

\begin{proof}
  Let $\alpha$ be an element of $P'(A)$, and let us suppose that it is a pure existential splitting of $P$ relative to $P'$. Now suppose that $\alpha \leq \exists_{{\pr_A}}(\beta)$, with $\beta$ element of $P(A\times B)$. By assumption, $\beta$ can be written as $\beta=\exists_{\pr_{A\times B}}(\gamma)$ with $\gamma $ element of $P'(A\times B\times C)$ and $\freccia{A\times B\times C}{\pr_{A\times B}}{A\times B}$. Hence,  since left adjoints compose, we have that
   \[\alpha\leq  \exists_{{\pr_A'}} (\  \exists_{\pr_{A\times B}}(\gamma)  \ ) = \exists_{{\pr_A}}(\gamma)\]
with $\freccia{A\times B}{\pr_A'}{A}$ and $\freccia{A\times B\times C}{\pr_A}{A}$.
  Since $\alpha$ is a pure existential splitting of $P$ relative to $P'$, then there exists an arrow $\freccia{A}{\angbr{f}{g}}{B\times C}$ such that 
  \[\alpha\leq P_{\angbr{\id_A}{f,g}}(\gamma).\]
  Then, since $\beta=\exists_{{\pr_{A\times B}}}(\gamma)$, and hence, $\gamma\leq P_{{\pr_{A\times B}}}(\beta)$, we deduce that
  \[\alpha \leq  P_{\angbr{\id_A}{f,g}}(\ P_{{\pr_{A\times B}}}(\beta) \ )  =P_{\angbr{\id_A}{f}}(\beta).\]
  Therefore, by Proposition~\ref{proposition:existential free elements}, we can conclude that $\alpha$ is a pure existential splitting. 
\end{proof}

Combining this lemma with the definition of pure existential cover, we obtain the following corollary:
\begin{corollary}\label{cor_rel_split_implies_cover}
  $P'$ is  a pure existential  cover of $P$ if and only if $P'$ is a pure existential relative cover of $P$.
\end{corollary}

\begin{proof}
  If $P'$ is a pure existential  cover of $P$, then it is in particular a pure existential relative cover of $P$ since  every pure existential splitting element of $P$ is obviously a pure existential splitting element of $P$ relative to $P'$.
  The converse follows  from Lemma~\ref{lemma_relative_splitting_and_cover_implies_splitting}  since every element of $P'$ which is a pure existential free element of $P$ relative to $P'$ is  also a pure existential splitting element for $P$.
\end{proof}

Now we are ready to recall the main result from~\cite{maietti_trotta}. Notice that, with respect to the original result, here we present an extra equivalent condition, based on Corollary~\ref{cor_rel_split_implies_cover}:

\begin{theorem}\label{theorem caract. gen. pure existential completion}
  The following are equivalent:  
  \begin{enumerate}  
    \item $P$ is isomorphic to the pure existential completion  $\purecompex{(P')}$ of a primary doctrine $P'$;
    \item $P$ satisfies the following points:
      \begin{enumerate}
        \item $P$ satisfies  the rule of choice RC;
        \item for every pure existential free element $\alpha$ and $\beta$ of $P(A)$, then $\alpha\wedge \beta$ is a pure existential free.
        \item $P$ has enough pure existential free elements;
       
      \end{enumerate}
    \item $P$ has a (unique) pure existential cover;
    \item $P$ has a (unique) pure existential relative cover.
  \end{enumerate}
\end{theorem}

A relevant application of the previous characterization is given in the context of doctrines with  \emph{Hilbert's}  $\epsilon$-\emph{operators}. We recall from  \cite{TECH}  the following definitions:
\begin{definition}
  Let $\doctrine{\mC}{P}$ be an elementary pure existential doctrine. An object $B$ of $\mC$ is equipped with Hilbert's  $\epsilon$-\textbf{operator} if, for any object $A$ in $\mC$ and any $\alpha$ in $P(A\times B)$ there exists an arrow $\freccia{A}{\epsilon_{\alpha}}{B}$ such that
  $\exists_{\pr_A}(\alpha)=P_{\angbr{\id_A}{\epsilon_{\alpha}}}(\alpha)$
  holds in $P(A)$.
\end{definition}

\begin{definition}\label{def_Hilbert_epsilon_doctrine}
  We say that an elementary pure existential doctrine $\doctrine{\mC}{P}$ is \textbf{equipped with}  Hilbert's $\epsilon$-\textbf{operators} if every object in $\mC$ is equipped with an  $\epsilon$-operator.
\end{definition}

Doctrines equipped with  Hilbert's $\epsilon$-operators have been characterized in  \cite{maietti_trotta} in terms of pure existential completions as follows:
\begin{theorem}\label{ex_doctrines_with_Hilber_are_ex_comp}
    Every elementary pure existential doctrine  is equipped with Hilbert's $\epsilon$-operators if and only if it is (equivalent to) the pure existential completion of itself. 
\end{theorem}

We conclude this section by recalling the following example from \cite{maietti_trotta}:

\begin{example}\label{example pure ex completion synt doct}
  Let $\theory_0$ be the  fragment $\lang_{=,\exists}$  of first-order Intuitionistic Logic, defined in Example~\ref{example regular fragment}, with no extra-logical axioms on an arbitrary signature  and let $\mathsf{H}_0$ be the Horn theory given the Horn fragment $\lang_=$ of first-order Intuitionistic Logic, defined in Example~\ref{example horn fragment}, with no extra-logical axioms on the same signature. Then, the elementary pure existential doctrine  $\doctrine{\cont}{\syntdoc_{=, \exists}^{\theory_0}} $   is the pure existential completion of the syntactic elementary doctrine
  $\doctrine{\cont}{\syntdoc_{=}^{\mathsf{H}_0}}$,
  namely $\purecompex{({\syntdoc_{=}^{\mathsf{H}_0}})}\equiv \syntdoc_{=,\exists}^{\theory_0}$.
\end{example}

\section{Regular and Exact completions of  elementary pure existential doctrines}

In this section we first recall well-known characterizations of the $\reg/\lex$ and $\ex/\lex$-completions in  \cite{SFEC,FECLEO,REC} and then we pass to remind notions and results related to  the regular and exact completions 
of elementary, pure existential doctrines in \cite{QCFF,EQC,UEC,TECH}.

Remember from \cite[Lem. 5.1]{SFEC} the following characterization of the  $\reg/\lex$ completion (we refer also \cite[Sec. A.1.3]{SAE_vol_1} for a detailed analysis of this completion):
\begin{theorem}[\cite{SFEC}]\label{carbonireg}
Any  regular category $\mathcal{A}$  is  the regular completion  of the full
subcategory $\mathcal{P}_{\mA}$ of  its regular projectives if and only if $\mathcal{P}_{\mA}$ is closed under finite limits in $\mathcal{A}$ and $\mathcal{A}$ has enough regular projectives, and in addition every object of $\mathcal{A}$ can be embedded in a regular projective.
\end{theorem}

Then, recall the following decomposition of $\ex/\lex$-completion shown by A.~Carboni and E.~Vitale~\cite{REC}:
\begin{theorem}[\cite{REC} ]\label{carbonidec}
For any category $\mathcal{C}$ with finite limits,
the $\ex/\lex$-completion of $\mathcal{C}$ is equivalent to the $\ex/\reg$-completion of the $\reg/\lex$-completion of $\mathcal C$, namely
$$\exlex{\mathcal{C}}\equiv \exreg{ \reglex{\mathcal{C}}}$$
\end{theorem}

\subsection{Regular and exact completions of doctrines}
We recall from \cite{TECH} the \emph{regular completion} of an elementary and pure existential doctrine. We provide a direct explicit description of this construction, while we refer to \cite{TECH} for its equivalent presentation  in terms of  the category of \emph{entire and functional relations} of  the completion  with comprehensions and comprehensive diagonals of an elementary existential doctrine.

\begin{definition}\label{remark_arrows_reg_espliciti}
Let $P$ be an elementary pure existential doctrine. 
The \textbf{ regular completion} $\regdoctrine{P}$  of $P$ is the category defined as follows:
\begin{itemize}
  \item an object is a pair $(A,\alpha)$ where $A$ is an object of $\mC$ and  $\alpha \in P(A)$;
\item an arrow from $(A,\alpha)$ to $(B,\beta)$ is given by an element $\phi$ of $P(A\times B)$ such that:
	\begin{enumerate}
		\item  $\phi \leq P_{\pr_1}(\alpha)\wedge P_{\pr_2}(\beta)$;\hfill \textbf{(well-defined)} 
		\item $\alpha\leq \exists_{\pr_1}(\phi)$; \hfill\textbf{(entire)} 
		\item  $P_{\angbr{\pr_1}{\pr_2}}(\phi)\wedge P_{\angbr{\pr_1}{\pr_3}}(\phi)\leq P_{\angbr{\pr_2}{\pr_3}}(\delta_B)$. \hfill \textbf{(functional)}
	\end{enumerate}
\end{itemize}
\end{definition}

\begin{remark}\label{remark_phi_is_exists_id_f_a}
  Observe that,  if $\psi$ and $\varphi$ are morphisms of $\regdoctrine{P}$ from $(A,\alpha)$ to $(B,\beta)$ and $\psi\leq \varphi$ then $\psi=\varphi$. We refer to \cite{pitts1981theory} for this remark.
\end{remark}
     
The universal properties of the regular completion of a doctrine are studied in~\cite[Thm. 3.3]{TECH}. We recall here the main result:
\begin{theorem}\label{theorem maietti rosolini pasquali regular comp}\label{mainreg}
  Let $\doctrine{\mC}{P}$ be an elementary and pure existential doctrine. Then the category $\regdoctrine{P}$ is regular, and the assignment $P\mapsto \regdoctrine{P}$ extends to a 2-functor
  \[\freccia{\EED}{\regdoctrine{-}}{\Reg}\]
  which is  a left biadjoint to the inclusion of the 2-category $\Reg$ of regular categories in the 2-category $\EED$ of elementary and pure existential doctrines given by the assignment $\mC\mapsto \Sub_{\mC}$.
\end{theorem}

\begin{remark}\label{rem_graph_functor}
  Recall from \cite{TECH} that given an elementary and pure existential doctrine $\doctrine{\mC}{P}$ and an arrow $\freccia{A}{f}{B}$ of $\mC$, we have that the \emph{graph} $P_{f \times \id_B}(\delta_B)$ of $f$ is an entire and functional relation from  $A$ to $B$ and this defines the \emph{graph functor} $\freccia{\mC}{G}{\regdoctrine{P}}$ which preserves finite products.
\end{remark}

\begin{remark}\label{remark_embedding_Gr_in_reg}
  Let $\doctrine{\mC}{P}$ be an elementary and pure existential doctrine. Notice that the categories $\Pred{P}$, $\mG_P$ and $\regdoctrine{P}$ have the same objects, but increasingly general morphisms.
  \end{remark}
Observe that the graph functor $\freccia{\mC}{G}{\regdoctrine{P}}$ defined in Remark~\ref{rem_graph_functor} extends to a functor from $\Pred{P}$ (and also
from $\mG_P$) to $\regdoctrine{P}$ and, more generally, to a functor from $\Pred{P'}$ (and also from $\mG_{P'}$) to $\regdoctrine{P}$ for a given subdoctrine $P'$.


\begin{definition}\label{def_graph_functor}
  Given an elementary pure existential doctrine $P$ and an elementary subdoctrine $P'$ we can define an embedding, called \textbf{graph functor}
  $$\freccia{\Pred{P'}}{G_{|_{P'}}}{\regdoctrine{P}}$$ by  mapping $(A,\alpha)$  of $ \Pred{P'}$ into  $(A,\alpha)$  of $\regdoctrine{P}$ and an arrow $\freccia{(A,\alpha)}{[f]}{(B,\beta)}$ of $\Pred{P'}$ into the arrow $ G_{|_{P'}}([f]) =P_{f \times \id_B}(\delta_B) \wedge\ (P_{\pi_1}(\alpha)\ \wedge\ P_{\pi_2 }(\beta)\ )$ from $(A,\alpha)$ to $(B, \beta)$ of $\regdoctrine{P}$.
\end{definition}

\begin{remark}\label{graph}
  The graph functor of an elementary pure existential doctrine $P$ with an elementary subdoctrine $P'$  can be also defined as  $G_{|_{P'}}([f])=\exists_{\angbr{\id_A}{f}}(\alpha)$ because $P$ has left adjoints along arbitrary arrows, see Remark~\ref{rem left adjoint elementary pure existential doc}.
\end{remark}

We refer to \cite{TECH} for the following result:
\begin{proposition}\label{prop_can_embedding}
  The previous assignments provide a well-defined functor $\freccia{\Pred{P'}}{G_{|_{P'}}}{\regdoctrine{P}}$, and it preserves finite limits. Moreover, it is faithful, and it induces a regular functor ${G_{|_{P'}}^{\reg}}$:
  \[\begin{tikzcd}
	  & \reglex{\Pred{P'}} \\
	  \Pred{P'} & \regdoctrine{P}
	  \arrow["{G_{|_{P'}}}"',from=2-1, to=2-2]
	  \arrow[hook,from=2-1, to=1-2]
	  \arrow["{G_{|_{P'}}^{\reg}}",from=1-2, to=2-2]
  \end{tikzcd}\]
\end{proposition}

\begin{proof}
  The first part follows by  \cite[Thm. 3.2]{TECH} (and from the fact that $\Pred{P'}$ is lex, see Remark~\ref{rem_finite_limit_pred}), while
  the existence of the regular functor ${G_{|_{P'}}^{\reg}}$ follows by the universal property of the $\reg/\lex$-completion in \cite{SFEC}.
\end{proof}

\begin{example}\label{ex_regular_comp_sub}
  The regular completion $\regdoctrine{\Sub_{\mC}}$ of the doctrine $\doctrine{\mC}{\Sub_{\mC}}$ of subobjects of a regular category is equivalent to $\mC$~\cite[Cor. 5.4]{TECH}.
\end{example}

\begin{example}\label{ex_regular_comp_weak_sub}
  The regular completion $\regdoctrine{\Psi_{\mD}}$ of the doctrine $\doctrine{\mD}{\Psi_{\mD}}$ of weak subobjects  presented in Example~\ref{example weak subobjects doctrine} coincides with the regular completion $\reglex{\mD}$ of the lex category $\mD$, in the sense of~\cite{REC}. We refer to \cite{TECH} for more details.
\end{example}

\begin{example}\label{ex_regular_synt_cat}
  The regular completion $\regdoctrine{\syntdoc_{=,\exists}^{\theory}}$  performed on the syntactic doctrine   $\doctrine{\cont}{\syntdoc_{=,\exists}^{\theory}}$ defined in Example~\ref{example regular fragment} provides exactly the syntactic category denoted $\mathcal{C}^{\mathsf{reg}}_{\theory}$ associated with the theory  $\theory$ 
  in \cite[Sec. D1.4]{SAE}. 
\end{example}

Combining the regular completion of an elementary pure existential doctrine with the exact completion $\exreg{-}$ of a regular category \cite{SFEC,REC}, one can define the so-called \emph{exact completion of an elementary and pure existential doctrine} as  pointed out in \cite[Sec. 3]{UEC}:
\begin{definition}\label{def_exc_comp_doc}
  Let $P$ be an elementary pure existential doctrine. 
  We call   the category $\exdoctrine{P}:=\exreg{\regdoctrine{P}}$ \emph{the  \textbf{exact completion}  of $P$}.
\end{definition}

The universal properties of the exact completion of an elementary and pure existential doctrine can be deduced by combining the universal properties of the regular completion of a doctrine (see Theorem~\ref{theorem maietti rosolini pasquali regular comp}) with the universal properties of the $\ex/\reg$-completion, see \cite[Cor. 3.4]{UEC}. We recall here the main result:
    
\begin{theorem}\label{theorem maietti rosolini exact comp}\label{mainreg}
  Let $\doctrine{\mC}{P}$ be an elementary and pure existential doctrine. The category $\exdoctrine{P}$ is exact, and the assignment $P\mapsto \exdoctrine{P}$ extends to a 2-functor
  \[\freccia{\EED}{\exdoctrine{-}}{\Excat}\]
  which is a left biadjoint to the inclusion of the 2-category $\Excat$ of exact categories in the 2-category $\EED$ of elementary and pure existential doctrines given by the assignment $\mC\mapsto \Sub_{\mC}$.
\end{theorem}

\begin{example}\label{example_ex_comp_weak_sub}
  The exact completion $\exdoctrine{\Psi_{\mD}}$ of the weak subobjects doctrine $\doctrine{\mD}{\Psi_{\mD}}$ presented in Example~\ref{example weak subobjects doctrine} coincides with the exact completion $\exlex{\mD}$ of the lex category $\mD$, in the sense of~\cite{REC}. We refer to \cite[Ex. 4.4]{UEC} for more details.
\end{example}

\begin{example}\label{example_exact_comp_synt_doct}
  The exact completion $\exdoctrine{\syntdoc_{=,\exists}^{\theory}}=\exreg{\regdoctrine{\syntdoc_{=,\exists}^{\theory}}}$ of the syntactic doctrine $\doctrine{\cont}{\syntdoc_{=,\exists}^{\theory}}$ coincides with the \emph{effectivization} $\mathcal{E}_{\theory}:=\mathbf{Eff}(\mC^{\mathsf{reg}}_{\theory})$
   of  the syntactic category $\mC^{\mathsf{reg}}_{\theory}$ in \cite[pp. 849-850]{SAE}. 
\end{example}

\begin{example}\label{ex_exact_comp_sub}
  The exact completion of the subobject doctrine $\doctrine{\mC}{\Sub_{\mC}}$ of a regular category $\mC$ coincides with the well known construction of  the exact completion of a regular category $\exreg{\mC}$ as observed in \cite{UEC}. 
\end{example}

\begin{remark}\label{rem_exreg_coincides_exlex}
  When $\mC$  is a regular category whose regular epimorphisms split, we have that $\Sub_{\mC}$ happens to be equivalent to the weak-subobject doctrine  (this happens because the classes of regular epis and monos define a factorization system for $\mC$, see for example \cite[Thm 2.1.3]{HCA2}).
   By combining this fact with previous examples, we obtain an abstract proof  that  $\reglex{\mC}\equiv \mC$ (by Examples~\ref{ex_regular_comp_sub} and \ref{ex_regular_comp_weak_sub}) and $\exreg{\mC}\equiv \exlex{\mC}$ (by Examples~\ref{ex_exact_comp_sub} and \ref{example_ex_comp_weak_sub}).  In \cite[Rem. 1.3.10(c)]{SAE_vol_1},  when $\mC$  is a regular category whose regular epimorphisms split, there is a direct proof of the fact that  $\reglex{\mC}\equiv \mC$,   and from this we can alternatively deduce that  $\exreg{\mC}\equiv \exlex{\mC}$ directly by Theorem~\ref{carbonidec}.
\end{remark}

\section{Characterization of  regular and exact completions of pure existential completions}\label{sec_main}
In this section we present our main results characterizing the regular and exact completions of pure existential completions of elementary doctrines (Theorem~\ref{thm_main_1} and Corollary~\ref{cor_ex_comp_main_2}).

 In particular, we  first show that 
 an elementary and pure existential doctrine $P$ is the pure existential completion of an elementary doctrine $P'$ if and only if 
 the canonical embedding of $\Pred{P'}$ into  $\regdoctrine{P}$ gives rise to an equivalence $\regdoctrine{P}\equiv \reglex{\Pred{P'}}$.
To this aim, we show that,  inside the regular completion $\regdoctrine{P}$ of  the pure existential completion $P$ of an elementary subdoctrine $P'$,
 the comprehension $(A, \alpha)$ of  a pure existential free element $\alpha$ of $P(A)$
is a regular projective.
To this purpose, we recall a standard, but useful lemma holding in every regular category. We refer to \cite[Sec. 4.3]{van1995basic}.

\begin{lemma}
  In a regular category $\mC$, an arrow $\freccia{A}{f}{B}$ is a regular epi if and only if the subobject doctrine  $\doctrine{\mC}{\Sub_{\mC}}$ satisfies $\top_B\leq \exists_{\pr_B} \Sub_{f\times \id_B}(\delta_B)$.
\end{lemma}

This lemma, combined with the definition of arrows in $\regdoctrine{P}$, allows us to provide a simple description of the regular epimorphisms of the regular completion $\regdoctrine{P}$ of an elementary pure existential doctrine $\doctrine{\mC}{P}$. We refer to \cite[Sec. 2.5]{pitts1981theory} and \cite{van1995basic} for more details.

\begin{lemma}\label{lem_regular_epis_in_reg_P}
  A morphism $\freccia{(A,\alpha)}{\phi}{(B,\beta)}$  of $\regdoctrine{P}$ is a regular epimorphism if and only if $\beta=\exists_{\pr_B}(\phi)$ in $P(B)$.
\end{lemma}

Now, we are going to show  that   pure existential splitting elements of a pure existential doctrine $P$  single out regular projective objects in the regular completion  $\regdoctrine{P}$.

To this purpose, we first prove the following useful lemma:
\begin{lemma}\label{lem_morph_from_ex_free_are_trakable} 
  If $\freccia{(A,\alpha)}{\phi}{(B,\beta)}$ is an arrow of $\regdoctrine{P}$ and $\alpha$ is a pure existential splitting element of $P$, then there exists an arrow $\freccia{A}{f}{B}$ such that $\alpha=P_{\angbr{\id_A}{f}}(\phi)$, with $\alpha\leq P_f(\beta)$. Moreover, for every arrow $\freccia{A}{g}{B}$ with such a property, we have that $\alpha \leq P_{\angbr{f}{g}}(\delta_B)$.
\end{lemma}
    \begin{proof} Let $\freccia{(A,\alpha)}{\phi}{(B,\beta)}$ be an arrow of in $\regdoctrine{P}$.
        By definition of arrows in $\regdoctrine{P}$, we have that $\alpha=\exists_{\pr_1}(\phi)$, and then, by the universal property of pure existential splittings, we can conclude that there exists an arrow $\freccia{A}{f}{B}$ such that $\alpha=P_{\angbr{\id_A}{f}}(\phi)$. Moreover, since $\phi\leq P_{\pr_2}(\beta)$, we can conclude that $\alpha \leq P_f(\beta)$.

        Now  let us consider another arrow $\freccia{A}{g}{B}$ such that $\alpha=P_{\angbr{\id_A}{g}}(\phi)$. By definition, we have that $\phi$ is functional in $P$, namely
         $$P_{\angbr{\pr_1}{\pr_2}}(\phi)\wedge P_{\angbr{\pr_1}{\pr_3}}(\phi)\leq P_{\angbr{\pr_2}{\pr_3}}(\delta_B).$$
         Then we can apply $P_{\angbr{\id_A,f}{g}}$ to both sides of this inequality, obtaining
         \[P_{\angbr{\id_A}{f}}(\phi) \wedge P_{\angbr{\id_A}{g}}(\phi)\leq P_{\angbr{f}{g}}(\delta_B) \]
         that is
         \[\alpha \leq P_{\angbr{f}{g}}(\delta_B).\]

    \end{proof}
    \begin{remark}\label{rem_uniqueness_in_Pred}
    Notice that, by Lemma~\ref{lem_morph_from_ex_free_are_trakable}, we have that every arrow $\freccia{(A,\alpha)}{\phi}{(B,\beta)}$ of $\regdoctrine{P}$ with $\alpha$  pure existential splitting induces a unique arrow in $\Pred{P}$.
        
    \end{remark}
    
    Furthermore, pure existential splitting elements single out  regular projectives as follows:

\begin{proposition}\label{prop_ex_sp_implies_proj}
Let $\doctrine{\mC}{P}$ be an elementary and pure existential doctrine. Then every object $(A,\alpha)$ where $\alpha$ is pure existential splitting is  a regular projective in  $\regdoctrine{P}$.
\end{proposition}
\begin{proof}
 Let us consider the following diagram
    \begin{equation}
\begin{tikzcd}\label{diag_proeittivi}
	&& {(C,\gamma)} \\
	\\
	{(A,\alpha)} && {(B,\beta)}
	\arrow["{\phi}"', from=3-1, to=3-3]
	\arrow["{\psi}", from=1-3, to=3-3]
 	\arrow["{\xi}", dashed, from=3-1, to=1-3]
\end{tikzcd}
\end{equation}
with $\psi$ regular epi in  $\regdoctrine{P}$. By Lemma~\ref{lem_regular_epis_in_reg_P} we know that  $\beta= \exists_{\pr_B}(\psi)$, with $\freccia{C\times B}{\pr_2}{B}$.  Hence, we  are going to use the fact that $\alpha$ is pure existential splitting to show that  there exists a morphism  $\xi$  of  $\regdoctrine{P}$ such that the diagram~\eqref{diag_proeittivi} commutes. 

From $\beta= \exists_{\pr_2}(\psi)$  and $\alpha \leq \exists_{\pr_1} (\phi) \leq \exists_{\pr_1}(\phi  \wedge P_{\pr_2}(\beta) )$ (since $\phi$ is entire and well-defined)
 we can deduce (combining these with  BCC and FR) that 
 $$\alpha\leq \exists_{\pr_1}(  P_{\angbr{\pr_1}{\pr_3}} (\phi)  \wedge  P_{\angbr{\pr_2}{\pr_3}}(\psi) ).$$
 where here projections have domain $A\times C\times B$.

 Therefore, since $\alpha$ is existential splitting, there exists a morphism $\freccia{A}{\angbr{\id_A}{  h_1,h_2}}{A \times C\times B}$  such that
\begin{equation}\label{eq_alpha_phi_psi}
\alpha \leq P_{\angbr{\id_A}{ h_2}}(\phi) \wedge P_{\angbr{h_1}{h_2}} (\psi).
 \end{equation}
 Then,  we can define $$\xi :=  G([h_1])=P_{ h_1 \times \id_C}({ \delta_C}) \wedge ( P_{\pr_1} (\alpha)  \wedge  P_{\pr_2} (\gamma)) $$
 after noting that  $\freccia{(A,\alpha)}{[h_1]}{(C,\gamma)}$  is an arrow of $\Pred{P}$  since $\alpha \leq P_{\angbr{h_1}{h_2}} (\psi)\leq P_{h_1}(\gamma)$ being $\psi$ 
 an arrow of $\regdoctrine{P}$.
 
 Now we show that  $\psi \circ \xi=\phi$ (where $\psi \circ \xi$ denotes the composition of morphisms in $\regdoctrine{P}$).
The fact that predicates are descent objects for the equality allows us to deduce that
  $$\psi \circ \xi=\exists_{\angbr{\pr_1}{\pr_3}}(  P_{\angbr{\pr_1}{\pr_2}} (G([h_1]) ) \wedge  P_{\angbr{\pr_2}{\pr_3}}(\psi)  )\leq P_{\pr_1}(\alpha) \wedge P_{h_1\times \id_B}(\psi)$$
  and by functionality of $\psi$, together with $ \alpha \leq  P_{\angbr{h_1}{h_2}}(\psi)$ by \eqref{eq_alpha_phi_psi}, we can deduce
  $$P_{\pr_1}(\alpha) \wedge   P_{h_1\times \id_B}(\psi)  \leq P_{{h_2}\times  \id_B} (\delta_B).$$
 Therefore, since  $\alpha\leq P_{\angbr{\id_A}{h_2}}(\phi)  $ by \eqref{eq_alpha_phi_psi}, and hence  $P_{\pr_1}(\alpha)  \leq P_{ \angbr{\pi_1}{ h_2 \pi_1} }(\phi)  $,
  we get that
  $$ \psi \circ \xi \leq \phi.$$
 Hence,  by Remark~\ref{remark_phi_is_exists_id_f_a}, we can conclude that  $\phi=\psi\circ\xi$, i.e.  the diagram \eqref{diag_proeittivi} commutes in $\regdoctrine{P}$. This concludes the proof that $(A,\alpha)$ is a regular projective.

\end{proof}
Recall that in the context of regular categories, we say that an object $A$ is \emph{covered} by a regular projective $B$ if there exists a regular epi $\freccia{B}{e}{A}$.

\begin{lemma}\label{lem_projective_cover}
  Let $\doctrine{\mC}{P}$ be an elementary and pure existential doctrine with enough pure existential free elements. Then every object $(B,\beta)$ of $\regdoctrine{P}$ is covered by a regular projective object $(A\times B,\alpha)$, with $\alpha$ pure existential splitting. 
\end{lemma}
\begin{proof}
  By definition of doctrine with enough pure existential free elements, for any element $\beta$ of $P(B)$ there exists a pure existential free element $\alpha$ of $P(A\times B)$ such that $\beta=\exists_{{\pr_2}}(\alpha)$ (hence, $\freccia{(A\times B,\alpha)}{[\pr_2]}{(B,\beta)}$ is a well-defined arrow of $\Pred{P}$).  Thus, for every object $(B,\beta)$ of $\regdoctrine{P}$, we can define in $\regdoctrine{P}$  the arrow $$\freccia{(A\times B,\alpha)}{G([{\pr_{2}}])}{(B,\beta)}$$
 where $\freccia{\Pred{P}}{G}{\regdoctrine{P}}$ is the graph functor defined in Definition~\ref{def_graph_functor} (with respect to $P$ itself), and it  is  a regular epi  of  $\regdoctrine{P}$ since by Remark~\ref{graph} 
   $G([{\pr_{2}}])= \exists_{\angbr{\id_{A\times B}}{ \pr_{2}}}(\alpha)$, which  implies
  $\exists_{\pr_3} (G([{\pr_{2}}])) =\exists_{\pr_3} ((\exists_{\angbr{\id_{A\times B}}{ \pr_{2}}}(\alpha))=\exists_{\pr_2}(\alpha)=\beta$.

   Finally, since every pure existential free element is in particular a pure existential splitting element, by Proposition~\ref{prop_ex_sp_implies_proj} we conclude that  $(A\times B,\alpha)$ is a regular projective object of $\regdoctrine{P}$ and it covers $(B,\beta)$.
\end{proof}
\begin{lemma}\label{lem_RC_implies_subobject_of _reg_proj}
Let $\doctrine{\mC}{P}$ be an elementary and pure existential doctrine. If $P$ satisfies the rule of choice then every object of $\regdoctrine{P}$ is a subobject of a regular projective.
\end{lemma}
\begin{proof}
If $P$ satisfies the rule of choice, then we have that every top element $\top_A$ is a pure existential splitting element (see Remark~\ref{rem_RC_and top_ex_splitting}), and hence $(A,\top_A)$ is a regular projective of $\regdoctrine{P}$  by Proposition~\ref{prop_ex_sp_implies_proj}. Therefore, every object $(A,\alpha)$ is a subobject of $(A,\top)$ in $\regdoctrine{P}$ via $G([\id_A])$ with $\freccia{(A,\alpha)}{[\id_A]}{(A, \top_A)} $, i.e. $\exists_{\angbr{\id_A}{\id_A}}(\alpha)$.

\end{proof}

By employing the previous results we can prove our main theorem:
\begin{theorem}\label{thm_main_1}
Let $\doctrine{\mC}{P}$ be an elementary and pure existential doctrine. Then $P$ is the pure existential completion of an elementary subdoctrine $P'$ if and only if the functor $\freccia{\reglex{\Pred{P'}}}{G_{|_{P'}}^{\reg}}{\regdoctrine{P}}$ provides an equivalence $\regdoctrine{P}\equiv \reglex{\Pred{P'}}$.

\end{theorem}
\begin{proof}
    $(\Rightarrow)$ Suppose that $P$ is the pure existential completion of an  elementary subdoctrine $P'$. To prove that $\freccia{\reglex{\Pred{P'}}}{G_{|_{P'}}^{\reg}}{\regdoctrine{P}}$ provides an equivalence $\regdoctrine{P}\equiv \reglex{\Pred{P'}}$, we will employ the characterization of the regular completion of a lex category
     as presented in \cite[Lem. 5.1]{SFEC} and recalled in Theorem~\ref{carbonireg}.

     In particular, we are going to show that  the image via ${G_{|_{P'}}}$ of $\Pred{P'}$ into $\regdoctrine{P}$ is a full subcategory of regular projectives, and that every object of $\regdoctrine{P}$ is covered (via a regular epi) by an object lying in the image of ${G_{|_{P'}}}$, and that every object of $\regdoctrine{P}$ is a subobject of an object lying in the image of ${G_{|_{P'}}}$.
    
Now, by Theorem~\ref{theorem caract. gen. pure existential completion} and Lemma~\ref{lemma_cover_P_and_ex_excomp}, we have that the pure existential free elements of $P$ are precisely the elements of $P'$ and of course, we have that every pure existential free is pure existential splitting.

Therefore, by Proposition~\ref{prop_ex_sp_implies_proj}, we have that every object $(A,\alpha)$ with $\alpha$ element of $P'(A)$ is a regular projective and, by Lemma~\ref{lem_projective_cover}, we can conclude that every object $(B,\beta)$ of $\regdoctrine{P}$ is covered by a regular projective object of the form $(A\times B,\alpha)$ with $\alpha$ pure existential free. 
Then, combining  Lemma~\ref{lem_RC_implies_subobject_of _reg_proj}  with the fact that every pure existential completion satisfies the rule of choice, we can conclude that every object  $(A,\alpha)$ of $\regdoctrine{P}$ is a subobject of a regular projective $(A,\top_A)$ (which is in the image of ${G_{|_{P'}}}$).
Finally,  the image of $\Pred{P'}$ via ${G_{|_{P'}}}$ into $\regdoctrine{P}$  is a lex full subcategory of $\regdoctrine{P}$ by Proposition \ref{prop_can_embedding} and Remark~\ref{rem_uniqueness_in_Pred}. 
By   the characterization of the regular completion in Theorem~\ref{carbonireg}, we conclude that $\freccia{\reglex{\Pred{P'}}}{G_{|_{P'}}^{\reg}}{\regdoctrine{P}}$  is an equivalence of categories.

   $(\Leftarrow)$ Let $\doctrine{\mC}{P'}$ be an elementary subdoctrine of $P$, and let us suppose the functor $\freccia{\reglex{\Pred{P'}}}{G_{|_{P'}}^{\reg}}{\regdoctrine{P}}$ provides an equivalence $\regdoctrine{P}\equiv \reglex{\Pred{P'}}$.

  It is immediate to observe that the fibre $\Sub_{\regdoctrine{P}}(A,\top_A)$ of the subobject doctrine of $\regdoctrine{P}$ is equivalent to  $P(A) $ 
  and that  the fibre $\Sub_{\regdoctrine{\Psi_{\Pred{P'}}}} ((A, \top_A),\id_A)$ of the subobject doctrine  of $\regdoctrine{\Psi_{\Pred{P'}}}$ is equivalent to  $\Psi_{\Pred{P'}}(A,\top_A) $. Hence, we conclude that $P(A)$   is equivalent to $\Psi_{\Pred{P'}} (A,\top_A)$ by the equivalence $\regdoctrine{P}\equiv \reglex{\Pred{P'}}$ induced by $G_{|_{P'}}^{\reg}$ (and by Example~\ref{ex_regular_comp_weak_sub}).

   Therefore, since from Example~\ref{example weak subobjects doctrine} we know that  every weak subobject doctrine is a full existential completion (and that every element of the fibre can be written as an existential quantifier of a top element), any $\gamma$ in $P(B)$ can be written as $\gamma=\exists_{f}(\alpha)= \exists_{{\pr_B}}(P_{f\times \id_B}(\delta_B)\wedge P_{{\pr_A}}(\alpha))$ for some $\alpha$ in $P'(A\times B)$ and some arrow $\freccia{A}{f}{B}$.

   Since $P'$ is elementary, then $P_{f\times \id_B}(\delta_B)\wedge P_{{\pr_A}}(\alpha)$ is an object of $P'$ and hence
  we conclude that every $\gamma$ can be written as $\gamma=\exists_{\pr_B}(\sigma)$ for some object $\sigma$ of $P'(A\times B) $.

    Now we show that every element $\alpha$ of $P'(A)$ is a pure existential splitting element of $P$ relative to $P'$
     (see Definition~\ref{definition pure ex splitting rel to P}).

    Suppose that $\alpha=\exists_{\pr_A}(\beta)$ in $P(A)$ with $\alpha$ element of $P'(A)$ and $\beta$  element of $P'(A\times B)$.
    Then, observe that  $\freccia{(A\times B, \beta) }{G_{|_{P'}}([\pr_A])}{(A,\alpha})$ 
    is a well defined arrow in $\Pred{P'}$ and that
     the arrow    $\freccia{(A\times B, \beta) }{G_{|_{P'}}([\pr_A])}{(A,\alpha})$ is a surjective epimorphisms and hence a regular epimorphism in $\regdoctrine{P}$ (as in the proof of Lemma~\ref{lem_projective_cover}).
    
     Since $(A,\alpha)$ is a regular projective, being in the image of ${G_{|_{P'}}^{\reg}}$ , there exists an arrow $\phi$ such that the diagram


  \begin{equation}\label{diag_regu_proj}
\begin{tikzcd}
	&& {(A\times B,\beta)} \\
	\\
	{(A,\alpha)} && {(A,\exists_{\pr_A}(\beta))}
	\arrow["{G_{|_{P'}}([\id_A])}"', from=3-1, to=3-3]
	\arrow["{G_{|_{P'}}([\pi_A])}", from=1-3, to=3-3]
	\arrow["\phi", dashed, from=3-1, to=1-3]
\end{tikzcd}
\end{equation}
commutes in $\regdoctrine{P}$. 

Then, by fullness of $G_{|_{P'}}^{\reg}$ since $\beta$ is an element of $P'$,  there exists a unique arrow $[\angbr{f_1}{f_2}]:(A,\alpha)\to (A\times B,\beta)$ of $\Pred{P'}$ such that $ \phi=G_{|_{P'}}([\angbr{f_1}{f_2}])$
and also $$\alpha\leq P_{\angbr{f_1}{f_2}}(\beta)$$
Hence we have that
  \begin{equation}\label{diag_regu_proj}
\begin{tikzcd}`
	&& {(A\times B,\beta)} \\
	\\
	{(A,\alpha)} && {(A,\exists_{\pr_A}(\beta))}
	\arrow["{G_{|_{P'}}([\id_A])}"', from=3-1, to=3-3]
	\arrow["{G_{|_{P'}}([\pi_1])}", from=1-3, to=3-3]
	\arrow["{G_{|_{P'}}([\angbr{f_1}{f_2}])}", dashed, from=3-1, to=1-3]
\end{tikzcd}
\end{equation}
commutes in $\regdoctrine{P}$ and by faithfulness of $G_{|_{P'}}$ we conclude
$$\alpha\leq P_{\angbr{f_1}{\id_A}}(\delta_A).$$
Combining this with $\alpha\leq P_{\angbr{f_1}{f_2}}(\beta)$  by the properties of equality we conclude 
$$\alpha\leq P_{\angbr{\id_A}{f_2}}(\beta).$$
This ends the proof that  any $\alpha$  of $P'$ is a pure existential splitting  element of $P$ relative to $P'$.  Furthermore, since any object of $P$ is (existentially) covered by an element of $P'$,  we get
that $P'$ is a pure  existential relative cover for $P$ and by Theorem~\ref{theorem caract. gen. pure existential completion} we finally conclude that $P$ is the pure existential completion of $P'$.

\end{proof}
Note that, by the universal property of the $\ex/\lex$-completion,  the graph functor  $\freccia{\reglex{\Pred{P'}}}{G_{|_{P'}}^{\reg}}{\regdoctrine{P}}$ extends to a functor $\freccia{\exlex{\Pred{P'}}}{G_{|_{P'}}^{\ex}}{\exdoctrine{P}}$:
\[\begin{tikzcd}
	\Pred{P'} & \reglex{\Pred{P'}} & \exlex{\Pred{P'}} \\
	& \regdoctrine{P} & \exdoctrine{P}.
	\arrow[hook, from=1-1, to=1-2]
	\arrow["G_{|_{P'}}^{\reg}"',from=1-2, to=2-2]
	\arrow[hook, from=2-2, to=2-3]
	\arrow[dashed,"G_{|_{P'}}^{\ex}",from=1-3, to=2-3]
	\arrow[hook, from=1-2, to=1-3]
\end{tikzcd}\]
Thus, we can extend our previous characterization to the exact completion of  elementary and pure existential doctrines as follows:
\begin{corollary}\label{cor_ex_comp_main_2}
    Let $\doctrine{\mC}{P}$ be an elementary and pure existential doctrine. Then, $P$ is the pure existential completion of an elementary sudoctrine $P'$ if and only if the functor $\freccia{\exlex{\Pred{P'}}}{G_{|_{P'}}^{\ex}}{\exdoctrine{P}}$ provides an equivalence $\exdoctrine{P}\equiv \exlex{\Pred{P'}}$. 
    
   \end{corollary}
   \begin{proof} 
    First observe that $G_{|_{P'}}^{\ex}$ is an equivalence if and only if $G_{|_{P'}}^{\reg}$ is an equivalence. This because, by definition, the restriction of $G_{|_{P'}}^{\ex}$ to $\reglex{\Pred{P'}}$ is precisely $G_{|_{P'}}^{\reg}$.
    Then, the result follows  by combining this fact with the definition of $\exdoctrine{P}$ (see Definition~\ref{def_exc_comp_doc}), the decomposition of the $\ex/\lex$-completion (see Theorem~\ref{carbonidec})  and  Theorem~\ref{thm_main_1}.\end{proof}

    The characterization of the regular and exact completions of doctrines  equipped with Hilbert's $\epsilon$-operators presented  in\cite[Thm. 6.2 (ii)]{TECH} can be seen now as a particular case of Theorem~\ref{thm_main_1} and Corollary~\ref{cor_ex_comp_main_2}. In fact, combining Theorem~\ref{ex_doctrines_with_Hilber_are_ex_comp} with these results we obtain the following corollary:
    
\begin{corollary}
Let $\doctrine{\mC}{P}$ be an elementary and pure existential doctrine. Then the following are equivalent:
\begin{itemize}
\item $P$ is equipped with Hilbert's $\epsilon$-operators;

\item the functor $\freccia{\reglex{\Pred{P}}}{{G}^{\reg}}{\regdoctrine{P}}$ provides an equivalence $\regdoctrine{P}\equiv \reglex{\Pred{P}}$;

\item the functor $\freccia{\exlex{\Pred{P}}}{G^{\ex}}{\exdoctrine{P}}$ provides an equivalence $\exdoctrine{P}\equiv \exlex{\Pred{P}}$. 
\end{itemize}

\end{corollary}

\begin{proof} 
It follows from Theorem~\ref{thm_main_1} and Corollary~\ref{cor_ex_comp_main_2}, since $P$ is isomorphic to the pure existential completion of itself $P^{\exists}$ by Theorem~\ref{ex_doctrines_with_Hilber_are_ex_comp}.

\end{proof}

Another corollary of our main results regards the presentation of the syntactic category ${\mathcal{C}^{\mathsf{reg}}_{\theory_0}}$  and of  its effectivization  $\mathcal{E}_{\theory_0}$ associated to a theory $\theory_0$  of   the fragment with  true constant, binary conjunctions, equality and existential quantifiers of first-order Intuitionistic Logic  with no extra-logical axioms,  as defined in \cite{SAE} (see Examples~\ref{ex_regular_synt_cat} and~\ref{example_exact_comp_synt_doct}). 
\begin{corollary}
Let $\theory_0$ be a regular theory  in the sense of \cite[Sec. D1.3]{SAE}, i.e. a theory of    the fragment   $\lang_{=,\exists}$  of first-order Intuitionistic Logic and no extra-logical axioms on a generic signature.
    Let $\mathsf{H}_0$ be the Horn theory given by
    the corresponding fragment   $\lang_{=}$  with no extra-logical axioms on the same signature. The syntactic category  ${\mathcal{C}^{\mathsf{reg}}_{\theory_0}}$  of $\theory_0$ 
     is equivalent to  the $\reg/\lex$-completion  $\reglex{\Pred{\syntdoc_{=}^{\mathsf{H}_0}}}$ of  the category of predicates of the syntactic doctrine  $\syntdoc_{ =}^{\mathsf{H}_0}$ of $\mathsf{H}_0$.  Hence,  also  
    its effectivization   $\mathcal{E}_{\theory_0}$
    is  the $\ex/\lex$-completion $\exlex{\Pred{\syntdoc_{=}^{\mathsf{H}_0}}}$ of the category of predicates of $\syntdoc_{=}^{\mathsf{H}_0}$.
        \end{corollary}
    \begin{proof}
    This follows from   Theorem~\ref{thm_main_1} and Corollary~\ref{cor_ex_comp_main_2} after recalling 
     from Examples~\ref{ex_regular_synt_cat} and~\ref{example_exact_comp_synt_doct} that ${\mathcal{C}^{\mathsf{reg}}_{\theory_0}}=\regdoctrine{\syntdoc_{=,\exists}^{\theory_0}}$ and that $\mathcal{E}_{\theory_0}=\exdoctrine{\syntdoc_{=,\exists}^{\theory_0}}$ and that $\syntdoc_{=,\exists}^{\mathsf{T}_0}$ is the pure existential completion
     of   $\syntdoc_{=}^{\mathsf{H}_0}$ as observed in Example \ref{example pure ex completion synt doct}.
    \end{proof}
    
    \begin{remark}\label{regularlogicwrongname}
    Despite the name \emph{regular theory} for a theory of  the fragment $(\top,\wedge,=,\exists)$-fragment of first-order Intuitionistic Logic   in~\cite[Sec. D1.3]{SAE}, the syntactic doctrine of such a  theory presented in Example~\ref{example regular fragment} does not coincide with the subobject doctrine of a regular category. Indeed, the fragment $(\top,\wedge,=,\exists)$-fragment of first-order Intuitionistic Logic   does not provide the internal language of regular categories,   which can be instead described by adopting a dependent type theory as that in \cite{MCBDTTCIPT} (a similar internal language is introduced there also for lex categories).
    The regular category  $\mC^{\mathsf{reg}}_{\theory}$ presented in \cite[pp. 849-850]{SAE} associated to such a dubbed regular theory (see also Example~\ref{ex_regular_synt_cat} ) is instead the regular completion $\regdoctrine{\syntdoc_{=,\exists}^{\theory}}$  of  the syntactic doctrine   $\doctrine{\cont}{\syntdoc_{=,\exists}^{\theory}}$.

    \end{remark}

Finally, a last relevant example of doctrines arising as pure existential completions is that of the so-called G\"odel hyperdoctrines presented in \cite{trotta23TCS}, arising in context of Dialectica interpretation. The original observation, for the more general case of fibrations, that a Dialectica fibration can be obtained combining the simple product and simple coproduct completions (i.e. the pure universal and pure existential completions in the case of doctrines) is due to P. Hofstra~\cite{hofstra2011}. 

By Corollary~\ref{cor_ex_comp_main_2}, we have that also every exact completion of a  G\"odel hyperdoctrine is an instance of the $\ex/\lex$-completion a category of predicates:

\begin{corollary}
Let $\doctrine{\mC}{P}$ be a  G\"odel hyperdoctrine (as defined in \cite{trotta23TCS}). Then we have the equivalences
\begin{itemize}
\item $\regdoctrine{P}\equiv \reglex{\Pred{P'}}$;
\item $\exdoctrine{P}\equiv \exlex{\Pred{P'}}$;
\end{itemize} 
 where $P'$ is the elementary subdoctrine of $P$ given by the pure existential free elements of $P$.
\end{corollary}
\section{A new description of Joyal's arithmetic universes}
 
Now, we apply our main results to the categorical setting of  \emph{Joyal's arithmetic universes} reported in \cite{JAULAP,vandijk2020godel}.  In \cite{JAULAP}
a more  general abstract notion of arithmetic universes in terms of list-arithmetic pretoposes is introduced.  Here, we provide a new description only for arithmetic universes in the sense of Joyal.
In the following, we refer to   \cite{JAULAP} for the  definition of \emph{predicates} on a \emph{Skolem theory} and of Joyal's arithmetic universes.

\begin{definition}
Let $\mS$ be a Skolem theory. The \emph{elementary doctrine of  \textbf{$\mS$-predicates}} is the functor $\doctrine{\mS}{\sR}$ sending an object $\mN^n$ into the poset $\sR(\mN^n)$ of predicates over $\mN^n$, namely the arrows $P: \mN^n \rightarrow \mN$ of the Skolem theory such that $P\cdot P=P$ where $\cdot $ is the multiplication of predicates (defined point-wise with the multiplication of natural numbers), and  where $P\leq Q$ is the point-wise order induced by natural numbers. The fibered equality  $\delta_N^n$ is defined via the equality of the Skolem theory, see \cite[Def.4.1]{JAULAP}.
\end{definition}

\begin{remark}
The category denoted by $\Pred{\mS}$ in \cite{JAULAP}  built by Joyal is a key inspiring example of  the category of predicates of an elementary doctrine introduced in \cite{TECH}. It was described in terms of free completions already in  \cite[Ex. 4.5]{UEC}. Using the language of doctrines, Joyal's category $\Pred{\mS}$  in \cite{JAULAP} is exactly the category of predicates  $\Pred{\sR}$ associated with the elementary doctrine $\doctrine{\mS}{\sR}$. 
\end{remark}

\begin{remark}
Let us  call $\mS_{in}^{\set}$   the embedding of the initial Skolem theory described in  \cite{JAULAP} within $\set$. Hence,  in this category every object is isomorphic to a finite product of $\mathsf{Nat}$ and the arrows of $\mS_{in}^{\set}$ are precisely \emph{the primitive recursive functions}. In this case, the fibres of the  elementary doctrine $\doctrine{(\mS_{in}^{\set})}{\mathsf{R}_{\mS_{in}^{\set}}}$ of $\mS_{in}^{\set}$-predicates can be equivalently be presented as follows:
\[\mathsf{R}_{\mS_{in}^{\set}}(\mathsf{Nat}^n)\simeq \{\freccia{\mathsf{Nat}^n}{f}{\mathsf{Nat}}\; | \; f\in \mS_{in}^{\set}(\mathsf{Nat}^n,\mathsf{Nat}) \mbox{ and } \forall m\in \mathsf{Nat}^n, \; f(m)=0 \mbox{ or } 1 \}\]
\end{remark}
In the following proposition we summarize some useful properties of the category $\Pred{\sR}$ associated with the doctrine $\doctrine{\mS}{\sR}$. We refer to \cite[Prop. 4.7]{JAULAP}  for more details.
\begin{proposition}\label{prop_epi_spli_arith_univ}
Given a Skolem theory $\mS$ and its elementary doctrine $\doctrine{\mS}{\sR}$ of $\mS$-predicates,  the category $\Pred{\sR}$ is regular and every regular epi splits.
\end{proposition} 
Now recall  the construction of  Joyal's arithmetic universes from~\cite[Def. 4.8]{JAULAP}:
\begin{definition}
Given a Skolem theory $\mS$,  an  \textbf{arithmetic universe} in the sense of \textbf{Joyal}  is  the category $\exreg{\Pred{\sR}}$.
\end{definition}
Then,  combining Corollary~\ref{cor_ex_comp_main_2} with Proposition~\ref{prop_epi_spli_arith_univ} we obtain the following result:
  \begin{corollary}\label{cor_1_arith_univ}
Every arithmetic universe   $\exreg{\Pred{\sR}}$ in the sense of Joyal  on a Skolem theory $\mS$ is equivalent to the exact completion 
$\exdoctrine{\sR^{\exists}}$  of the pure existential completion $\sR^{\exists}$ of the elementary doctrine $\doctrine{\mS}{\sR}$ of $\mS$-predicates.
\end{corollary}
\begin{proof}
In the category $\Pred{\sR}$ we have that regular epimorphisms split by Proposition~\ref{prop_epi_spli_arith_univ} and hence from  Remark~\ref{rem_exreg_coincides_exlex} we derive that $\exreg{\Pred{\sR}}\equiv \exlex{\Pred{\sR}}$. Finally, by applying Corollary~\ref{cor_ex_comp_main_2} we conclude that the arithmetic universe $\exreg{\Pred{\sR}}$  is equivalent to the exact completion of the pure existential completion $\sR^{\exists}$ of  the elementary doctrine $\sR$:
\[\exreg{\Pred{\sR}} \equiv \exdoctrine{\sR^{\exists}}.\]
\end{proof}

\begin{corollary}\label{cor_arith_univ_su_N}
 The initial arithmetic universe $\exreg{\Pred{\mathsf{R}_{\mS_{in}^{\set}}}}$  on the initial Skolem theory embedded in $\set$
  is equivalent to the exact completion $\exdoctrine{\mathsf{R}_{\mS_{in}^{\set}}^{\exists}}$ where the elements of the fibre $\mathsf{R}_{\mS_{in}^{\set}}^{\exists}(\mathsf{Nat})$  are exactly the recursive enumerable subsets of $\mathsf{Nat}$ in $\set$.
\end{corollary}
\begin{proof}
By Corollary~\ref{cor_1_arith_univ} we have that
\[\exreg{\Pred{\mathsf{R}_{\mS_{in}^{\set}}}} \equiv \exdoctrine{\mathsf{R}_{\mS_{in}^{\set}}^{\exists}}.\]
Then, observe that the fibres of  the pure existential completion $\doctrine{(\mS_{in}^{\set})}{\mathsf{R}_{\mS_{in}^{\set}}^{\exists}}$ of the elementary doctrine $\doctrine{(\mS_{in}^{\set})}{\mathsf{R}_{\mS_{in}^{\set}}}$  are exactly the \emph{recursively enumerable predicate} because, by Theorem~\ref{theorem caract. gen. pure existential completion}, every element of $\mathsf{R}_{\mS_{in}^{\set}}^{\exists}(\mathsf{Nat})$ 
 can be written as an existential quantifier of a primitive recursive predicate, and it is well-known that every recursively enumerable predicate can be proved to be presented as an existential quantifier of a primitive recursively enumerable predicate, for example, from  \cite[Thm. II.1.8, Thm. I.3.3, Ex I.2.8]{Soare78}.
\end{proof}

\section{Conclusion}
We have provided a new description of Joyal's arithmetic universes \cite{JAULAP}  as an application of  a characterization of regular and  exact completions of  pure existential completions of elementary doctrines. This characterization  extends a previous one proved in \cite{TECH} for doctrines equipped with Hilbert's $\epsilon$-operators.

In particular, we have proved that for elementary doctrines arising as pure existential completions, their  regular and exact completions happen to be equivalent to the $\reg/\ex$ and $\ex/\lex$-completions, respectively, of the category of predicates associated with the subdoctrine of their pure existential free elements. To reach this goal, we took advantage of the intrinsic characterization of doctrines arising as pure existential completions presented in \cite{maietti_trotta}, slightly extended with another equivalent presentation here.

Using these results we have deduced that an arithmetic universe in the sense of Joyal can be seen  as  the  exact completion of the pure existential completion of the doctrine of predicates of its Skolem theory.  In particular,  the  initial arithmetic universe in the standard category of ZFC-sets turns out to be the completion with exact quotients of the doctrine of recursively enumerable predicates.

Other examples of application of our characterization include the  so called syntactic category  in \cite{SAE} associated to the so called regular fragment of first-order logic in \cite{SAE} (and its effectivization) and the regular and exact completion of a G\"odel hyperdoctrine~\cite{trotta23TCS,trotta2022}.

As future work, we aim to extend our results to regular and exact completions of other classes of doctrines obtained as generalized existential completions, including the case of the {\it full existential completion} of primary doctrines in the sense of \cite{maietti_trotta},  as initiated in \cite{maietti_trotta_arxiv},  with applications to sheaf theory.

\section*{Acknowledgements}

We acknowledge fruitful conversations with  Samuele Maschio, Manlio Valenti, Fabio Pasquali and Pino Rosolini on topics presented in this paper. 
Finally, we thank the anonymous referees for their valuable comments. The first author acknowledges to be a member of INdaM-Gnsaga.

\bibliographystyle{plain}
\bibliography{biblio_reg}

\end{document}